\documentclass[paper=a4, fontsize=11pt, notitlepage]{article}

\usepackage[english]{babel}
\usepackage{amsmath,amsfonts,amsthm, amssymb}
\usepackage[margin=1.5in]{geometry}
\usepackage{graphicx}
\usepackage{float}
\usepackage[]{xcolor}
\usepackage{bm}
\usepackage{enumitem}

\usepackage{tikz}
\usetikzlibrary{backgrounds}
\usetikzlibrary{arrows}
\usetikzlibrary{shapes,shapes.geometric,shapes.misc}

\newtheorem{thm}{Theorem}[section]
\newtheorem{prop}[thm]{Proposition}
\newtheorem{lem}[thm]{Lemma}
\newtheorem{cor}[thm]{Corollary}
\newtheorem{rem}[thm]{Remark}

\newtheorem{con}[thm]{Conjecture}
\newtheorem{ques}[thm]{Question}

\theoremstyle{definition}
\newtheorem{defn*}{Definition}

\DeclareMathOperator{\diam}{diam}
\DeclareMathOperator{\rad}{rad}
\DeclareMathOperator{\eqt}{eqt}

\DeclareMathOperator{\dd}{D}

\pgfdeclarelayer{edgelayer}
\pgfdeclarelayer{nodelayer}
\pgfsetlayers{background,edgelayer,nodelayer,main}
\tikzstyle{none}=[inner sep=0mm]
\tikzstyle{whitevertex}=[fill=white, draw=black, shape=circle]
\tikzstyle{blackvertex}=[fill=black, draw=black, shape=circle]
\tikzstyle{dgreyvertex}=[fill=gray, draw=black, shape=circle]
\tikzstyle{lgreyvertex}=[fill={rgb,255: red,191; green,191; blue,191}, draw=black, shape=circle]

\tikzstyle{thick}=[-, line width=1.4pt]
\tikzstyle{light}=[-, draw={rgb,255: red,191; green,191; blue,191}]

\title{\huge Isometric Cycles and a Generalization of Moore Graphs}
\author{Brandon Du Preez \\
	University of Cape Town\\
	Department of Mathematics and Applied Mathematics\\
	Laboratory for Discrete Mathematics and Theoretical Computer Science\\
brandon.dupreez@uct.ac.za}
\date{July 2024}

\begin{document}
	\maketitle
	
	\begin{abstract}
		The equator of a graph is the length of a longest isometric cycle.
		We bound the order $n$ of a graph from below by its equator $q$, girth $g$ and minimum degree $\delta$ --- and show that this bound is sharp when there exists a Moore graph with girth $g$ and minimum degree $\delta$.
		The extremal graphs that attain our bound give an analogue of Moore graphs. 
		We prove that these extremal `Moore-like' graphs are regular, and that every one of their vertices is contained in some maximum length isometric cycle. 
		We show that these extremal graphs have a highly structured partition that is unique, and easily derived from any of its maximum length isometric cycles.
		We characterize the extremal graphs with girth 3 and 4, and those with girth 5 and minimum degree 3.
		We also bound the order of $C_4$-free graphs with given equator and minimum degree, and show that this bound is nearly sharp.
		We conclude with some questions and conjectures further relating our extremal graphs to cages and Moore graphs.
	\end{abstract}
	
	\begin{center}
		\textbf{Keywords}\\
		Isometric cycle, minimum degree, equator, order bound, Moore graph
	\end{center}
	
	\section{Introduction}
	The \textit{equator} of a graph $G$, denoted $\eqt(G)$, is the length of the longest isometric cycle of $G$.
	Lokshtanov proved that, somewhat surprisingly, the equator can be computed in polynomial time \cite{lokshtanov2009finding}. 
	If $G$ has diameter $d$ and radius $r$, then two opposite vertices $u$ and $v$ of a longest isometric cycle satisfy $d(u,v) \leq d$, so the equator is bounded by $\eqt(G) \leq 2d + 1 \leq 4r + 1$.
	A large body of research (discussed in the next paragraph) tackles the problem of bounding the order of a graph by some combination of its radius, diameter, girth and minimum degree.
	Given the relationship between the equator, radius and diameter, we expect that some of these bounds could be improved for graphs with large equators.
	In this paper, we bound the order of a graph from below using its equator, girth, and minimum degree. 
	This bound generalizes the well-known Moore bound for a graph of given girth and minimum degree (see \cite{miller2012moore} for a detailed survey).
	We describe the resulting `Moore-like' graphs that attain this bound, and characterize them exactly for the low girth cases.
	These extremal graphs are highly structured, and reflect some of the properties of the classic Moore graphs.
	
	We briefly survey the existing literature on similar order bounds.
	The problem of bounding the diameter and radius using the minimum degree and order was first addressed by Erd\"{o}s, Pach, Pollack and Tuza in \cite{erdHos1989radius}.
	They proved the following theorems:
	\begin{thm}\textup{\cite{erdHos1989radius}}
		\label{thm:intro_eppt_any}
		Let $G$ be a connected graph with minimum degree $\delta \geq 2$, order $n$, radius $r$ and diameter $d$. Then
		\begin{align*}
			d \leq \left\lfloor  \frac{3n}{\delta + 1} \right\rfloor - 1 \text{ and }
			r \leq \frac{3(n-3)}{2(\delta + 1)} + 5.
		\end{align*}
	\end{thm}
	\begin{thm}\textup{\cite{erdHos1989radius}}
		\label{thm:intro_eppt_trianglefree}
		Let $G$ be a connected triangle-free graph with minimum degree $\delta \geq 2$, order $n$, radius $r$ and diameter $d$. Then
		\begin{align*}
			d \leq 4 \left\lceil  \frac{n - \delta - 1}{2\delta} \right\rceil \text{ and }
			r \leq \frac{n-2}{\delta} + 12.
		\end{align*}
	\end{thm}
	The radius bound of Theorem \ref{thm:intro_eppt_any} was improved in \cite{kim2012maximum} by Kim, Rho, Song and Hwang. 
	The bounds in Theorem \ref{thm:intro_eppt_any} can be re-arranged to bound $n$ by roughly
	\begin{align}
		n \geq d\left( \frac{\delta + 1}{3} \right) + \mathcal{O}(\delta) \text{ and } n \geq 2r\left( \frac{\delta + 1}{3} \right) - \mathcal{O}(\delta),
	\end{align}
	and similarly from Theorem \ref{thm:intro_eppt_trianglefree} we get roughly
	\begin{align}
		n \geq d \left( \frac{2\delta}{4} \right) + \mathcal{O}(\delta) \text{ and } n \geq 2r \left( \frac{2\delta}{4} \right) - \mathcal{O}(\delta).
	\end{align}
	The radius and diameter of $C_4$-free graphs was also bounded in \cite{erdHos1989radius}.
	The bounds are very close to being sharp for infinitely many values of $\delta$.
	These bounds can be re-arranged to yield the following:
	\begin{align}
		\label{eqn:c_4_free_erdos}
		n \geq d \left( \frac{\delta^2 - 2\lfloor \frac{\delta}{2} \rfloor + 1}{5} \right) \text{ and } n \geq 2r \left( \frac{\delta^2 - 2\lfloor \frac{\delta}{2} \rfloor + 1}{5} \right).
	\end{align}
	
	In the same vein, Alochukwu and Dankelmann bound the diameter and radius of graphs with girth at least 6, as well as $\{C_4, C_5\}$-free graphs in \cite{alochukwu2021distances}.
	\begin{thm}\textup{\cite{alochukwu2021distances}}
		\label{thm:intro_aa_girth6}
		Let $G$ be a graph of order $n$, diameter $d$, minimum degree $\delta$ and girth at least 6. Then
		\begin{align*}
			d \leq \frac{3n}{\delta^2 -\delta + 1} - 1 \text{ and } r \leq \frac{3n}{2(\delta^2 -\delta + 1)} + 10.
		\end{align*}
	\end{thm}
	The bounds in Theorem \ref{thm:intro_aa_girth6} are very nearly sharp for infinitely many values of $\delta$.
	Again, we can interpret this result as a bound on $n$ to get roughly
	\begin{align}
		n \geq d\left( \frac{2\delta^2 - 2\delta + 2}{6} \right) + \mathcal{O}(\delta^2) \text{ and } n \geq 2r\left( \frac{2\delta^2 - 2\delta + 2}{6} \right) - \mathcal{O}(\delta^2).
	\end{align}
	Focusing on a radius bound, Dvor\'{a}k, Hintum, Shaw and Tiba provide the most general result so far of this kind in \cite{dvorak2022radius}. 
	\begin{thm}\textup{\cite{dvorak2022radius}}
		\label{thm:intro}
		Let $G$ have even girth $g = 2k$ and radius $r$. Then
		\begin{align*}
			r \leq \frac{nk(\delta - 2)}{2((\delta - 1)^k - 1)} + 3k.
		\end{align*} 
	\end{thm}
	Notice that this bound is approximately $r \leq \frac{ng}{2M}$, where $M$ is the Moore bound for even girth $g$ and minimum degree $\delta$. 
	The bound is very nearly sharp for girths 6, 8 and 12.
	
	Other variations on the themes above have been studied.
	The diameter of a graph was bounded in terms of its order, and the number of different degrees of its vertices in \cite{mukwembi2012note}.
	The diameter has also been bounded in terms of minimum degree and chromatic number in \cite{czabarka2009diameter} and \cite{czabarka2021maximum}.
	The circumference (which is the length of a longest cycle) was bounded below by the minimum degree and girth in \cite{ellingham2000girth}, and both the circumference and girth were bounded by radius and diameter in \cite{qiao2018maximum}.
	
	The rest of the paper is organized as follows.
	In Section \ref{sec:lower_bound}, we bound the order of a graph from below using its equator, girth, and minimum degree. 
	In graphs where the equator is large compared to the diameter ($q \approx 2d$) and radius ($q \approx 4r$), the bound we present in Section \ref{sec:lower_bound} compares favorably against all the above bounds from the literature.
	We also bound the order of $C_4$-free graphs with given equator and minimum degree.
	Section \ref{sec:sharpness} demonstrates that our equator-girth-degree bound is sharp for all combinations of girth and minimum degree for which there exists a Moore graph, and that the $C_4$-free bound is almost sharp.
	The graphs that meet our bound can be though of as `equator-constrained' generalizations of Moore graphs, and we call them \textit{equatorial graphs}.
	We describe the structure of these equatorial graphs in Section \ref{sec:equatorial_structure}, present results concerning the existence of equatorial graphs in Section \ref{sec:extremal_parameters}, and characterise them exactly for some low girth cases in Section \ref{sec:low_girth_char}. 
	We prove that, like Moore graphs, equatorial graphs are regular and every one of their vertices lies on a maximum length isometric cycle.
	Further, every equatorial graph admits a unique partition induced by its isometric cycles.
	
	\section{Definitions and Notation}
	Standard definitions (such as order, diameter, connectivity and so on) can be found in \cite{DiestelIII} and \cite{BondyMurtyGT}.
	For all definitions here, consider a connected graph $G = (V,E)$.
	A graph is $C_4$-\textit{free} if it does not contain any 4-cycle (induced or otherwise).
	We use the convention that $n = |V|$ is the order of $G$ and $\delta$ is the minimum degree of $G$.
	The \textit{girth} of $G$ is the length of a shortest cycle in $G$.
	The \textit{radius} and \textit{diameter} of $G$ are denoted $\rad(G)$ and $\diam(G)$ respectively.
	We use the convention that an acyclic graph has girth $\infty$ and equator $0$.
	We will often let $r, d, g$ and $q$ stand for radius, diameter, girth and equator, respectively.
	Thus, a maximum length isometric cycle of $G$ will be called an \textit{isometric} $q$-cycle.
	Further, we will usually let $k = \lceil \frac{g}{2} \rceil - 1$.
	The \textit{disk of radius} $i$ around vertex $u$ is $\dd_i(u) = \{v\in V: d(u,v) \leq i\}$.
	For an edge $e=uv$, denote by $\dd_i(e) = \dd_i(u) \cup \dd_i(v)$ the vertices within distance $i$ of $e$.
	Let the $i$-\textit{degree} of $v$ be $d_i(v) = |\dd_i(v)|$, and similarly denote $d_i(e) = |\dd_i(e)|$.
	The degree of a vertex $v$ is $d(v)$, and is not the same as the 1-degree $d_1(v) = 1 + d(v)$.
	For positive integers girth $g \geq 3$, degree $\delta \geq 2$, and $k = \lceil \frac{g}{2} \rceil - 1$, denote the \textit{Moore Bound} by
	\begin{align*}
		M(\delta, g) = 
		\begin{cases}
			\displaystyle 1 + \sum_{i=0}^{k-1} \delta(\delta - 1)^i = \frac{\delta(\delta - 1)^k - 2}{\delta - 2} & \text{ when } g \text{ is odd}.\\
			\displaystyle 2 + \sum_{i=1}^{k} 2(\delta - 1)^i = \frac{2(\delta - 1)^{k+1} - 2}{\delta - 2} & \text{ when } g \text{ is even}.
		\end{cases}
	\end{align*}
	A graph with girth $g$, minimum degree $\delta$ and order $M(\delta, g)$ is called a \textit{Moore graph}.
	If a graph has minimum degree $\delta$ and girth $g$, it is a $(\delta+, g)$-\textit{graph}, and if it is regular we call it a $(\delta, g)$-\textit{graph}.
	A $(\delta, g)$-\textit{cage} is a $(\delta, g)$-graph of minimum order, and we denote by $C(\delta, g)$ the order of a $(\delta, g)$-cage.
	A graph with minimum degree $\delta$, girth $g$ and equator $q$ is a $(\delta+, g, q)$-\textit{graph}, and if it is $\delta$-regular, it is a $(\delta, g, q)$-\textit{graph}.
	
	\section{The Lower Bound}
	\label{sec:lower_bound}
	In this section, we prove a lower bound on the order of a graph in terms of its girth, minimum degree, and the length of a longest isometric cycle.
	By an application of the same technique, we bound the order of a $C_4$-free graph using its minimum degree and equator.
	The bounds we prove here are similar to the bounds mentioned in the introduction --- with the equator $q$ replacing either the diameter $d$ or twice the radius $2r$.
	
	Remember we denote $k = \lceil \frac{g}{2} \rceil - 1$.
	Note that we can bound the minimum $k$-degree of a vertex / edge if we know the minimum degree and the girth (see Remark \ref{rem:k_degree}).
	To do so, consider a breadth-first search tree rooted at a single vertex for odd girth, and rooted at an edge for even girth (see Figure \ref{fig:bfs_tree}).
	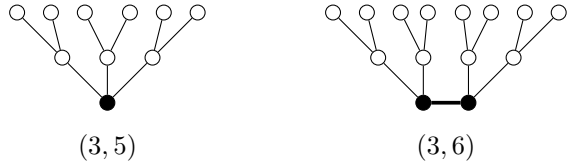
\begin{figure}[h]
		\centering
		\begin{tikzpicture}[scale=0.6, inner sep= 0.7mm]
	\begin{pgfonlayer}{nodelayer}
		\node [style=blackvertex] (0) at (-3, 0) {};
		\node [style=whitevertex] (1) at (-4, 1) {};
		\node [style=whitevertex] (2) at (-3, 1) {};
		\node [style=whitevertex] (3) at (-2, 1) {};
		\node [style=whitevertex] (4) at (-4.25, 2) {};
		\node [style=whitevertex] (5) at (-5, 2) {};
		\node [style=whitevertex] (6) at (-3.5, 2) {};
		\node [style=whitevertex] (7) at (-2.5, 2) {};
		\node [style=whitevertex] (8) at (-1.75, 2) {};
		\node [style=whitevertex] (9) at (-1, 2) {};
		\node [style=blackvertex] (10) at (4, 0) {};
		\node [style=blackvertex] (11) at (5, 0) {};
		\node [style=whitevertex] (12) at (4, 1) {};
		\node [style=whitevertex] (13) at (3, 1) {};
		\node [style=whitevertex] (14) at (5, 1) {};
		\node [style=whitevertex] (15) at (6, 1) {};
		\node [style=whitevertex] (16) at (4.1, 2) {};
		\node [style=whitevertex] (17) at (2, 2) {};
		\node [style=whitevertex] (18) at (3.5, 2) {};
		\node [style=whitevertex] (19) at (2.75, 2) {};
		\node [style=whitevertex] (20) at (4.875, 2) {};
		\node [style=whitevertex] (21) at (7, 2) {};
		\node [style=whitevertex] (22) at (6.25, 2) {};
		\node [style=whitevertex] (23) at (5.5, 2) {};
		\node [style=none] (24) at (-3, -1) {$(3,5)$};
		\node [style=none] (25) at (4.5, -1) {$(3,6)$};
	\end{pgfonlayer}
	\begin{pgfonlayer}{edgelayer}
		\draw (1) to (0);
		\draw (2) to (0);
		\draw (0) to (3);
		\draw (5) to (1);
		\draw (4) to (1);
		\draw (6) to (2);
		\draw (7) to (2);
		\draw (8) to (3);
		\draw (3) to (9);
		\draw [style=thick] (10) to (11);
		\draw (12) to (10);
		\draw (13) to (10);
		\draw (14) to (11);
		\draw (15) to (11);
		\draw (17) to (13);
		\draw (19) to (13);
		\draw (18) to (12);
		\draw (16) to (12);
		\draw (20) to (14);
		\draw (22) to (15);
		\draw (15) to (21);
		\draw (23) to (14);
	\end{pgfonlayer}
\end{tikzpicture}
		\caption{Left is a breadth-first search tree rooted at the bold vertex in a $(3,5)$-graph, and right is one rooted at the bold edge in a $(3,6)$-graph.}
		\label{fig:bfs_tree}
	\end{figure} 
	
	\begin{rem}
		\label{rem:k_degree}
		Let $G=(V,E)$ be a graph with girth $g$ and minimum degree $\delta$, and let $k = \lceil \frac{g}{2} \rceil - 1$.
		If $g$ is odd, then $d_k(v) \geq M(\delta, g)$ for every vertex $v$.
		If $g$ is even, then $d_k(e) \geq M(\delta, g)$ for every edge $e$.
	\end{rem} 
	
	We establish the lower bound for graphs of odd girth by counting, in two ways, the number of pairs $(v, D)$ where $v$ is some vertex of the graph, and $D$ is the disk of radius $k$ centered at some vertex of an isometric $q$-cycle.
	
	\begin{lem}
		\label{lem:low_bound_odd}
		Let $G=(V,E)$ be a graph of order $n$, equator $q$, odd girth $g = 2k+1$ and minimum degree $\delta$. If $q > 6k + 3$, then
		\begin{align*}
			n \geq \frac{q}{g}M(\delta, g).
		\end{align*}
	\end{lem}
	
	\begin{proof}
		Let $C = u_0, u_1, \dots, u_{q-1}$ be an isometric $q$-cycle, and note by Remark \ref{rem:k_degree} that $d_k(u_i) \geq M(\delta, g)$ for all $u_i \in C$.
		Since $C$ is an isometric cycle, we have that $\dd_k(u_i) \cap \dd_k(u_j) \neq \emptyset$ only if $d_C(u_i, u_j) \leq 2k$.
		Thus each vertex $v$ of $G$ belongs to at most $2k+1 = g$ disks of the form $\dd_k(u_i)$ ($0\leq i \leq q-1$).
		Since each disk has at least $M(\delta, g)$ vertices, and each vertex is counted in at most $g$ disks, we have
		\begin{align*}
			n &\geq \left| \bigcup_{i=0}^{q-1} \dd_k(u_i) \right| \geq \frac{1}{g} \sum_{i=0}^{q-1} |\dd_k(u_i)| \geq \frac{q}{g} M(\delta, g),
		\end{align*}
		which completes the proof.
	\end{proof}
	
	The proof in the even girth case is similar: we count the number of vertices appearing in the disks $\dd_k(e)$, where $e$ is an edge of an isometric $q$-cycle.
	
	\begin{lem}
		\label{lem:low_bound_even}
		Let $G=(V,E)$ be a graph of order $n$, equator $q$, even girth $g = 2k+2$ and minimum degree $\delta$. If $q > 6k + 3$, then
		\begin{align}
			n \geq \frac{q}{g}M(\delta, g).
		\end{align}
	\end{lem}
	
	\begin{proof}
		Let $C = u_0, \dots, u_{q-1}$ be an isometric $q$-cycle, and let $e_i = u_iu_{i+1}$ (subscripts taken mod $q$, so $e_{q-1} = u_{q-1}u_0$).
		As $C$ is isometric, we have that $\dd_k(e_i) \cap \dd_k(e_j) = \emptyset$ only if $d_C(e_i, e_j) \leq 2k$. 
		Thus $v$ belongs to at most $2k+2 = g$ disks of the form $\dd_k(e_i)$.
		By Remark \ref{rem:k_degree}, the $k$-degree of each edge $e_i$ is bounded by $d_k(e_i) \geq M(\delta, g)$. 
		Counting the vertices in each disk yields 
		\begin{align*}
			n &\geq \left| \bigcup_{i=0}^{q-1} \dd_k(e_i) \right| \geq \frac{1}{g} \sum_{i=0}^{q-1} |\dd_k(e_i)| \geq \frac{q}{g} M(\delta, g),
		\end{align*}
		which completes the proof.
	\end{proof}
	
	Combine Lemmas \ref{lem:low_bound_odd} and \ref{lem:low_bound_even} to get a simple unified bound that holds for all $(\delta+, g, q)$-graphs.
	
	\begin{thm}
		\label{thm:lower_bound}
		Let $G$ be a graph with order $n$, girth $g$, $k = \lceil\frac{g}{2}\rceil - 1$, minimum degree $\delta$ and equator $q > 6k + 3$. Then
		 \begin{align}
		 		\label{eqn:lower_bound}
		 		n \geq \frac{q}{g}M(\delta, g).
		 \end{align}
	\end{thm}
	
	The bound in Theorem \ref{thm:lower_bound} for $g=3$ is essentially the bound for a generic graph with no girth restriction (as every graph with a cycle has girth at least 3). 
	In this case, we have $n\geq \frac{q(\delta + 1)}{3}$.
	For a graph $G$ with $\eqt(G) \approx 2\diam(G)$, Theorem \ref{thm:lower_bound} roughly doubles the lower bound of Theorem \ref{thm:intro_eppt_any} in \cite{erdHos1989radius}.
	Similarly, when $\eqt(G) \approx 4\rad(G)$, the bound provided by Theorem \ref{thm:lower_bound} is about twice the bound in Theorem \ref{thm:intro_eppt_any}.
	
	If $G$ is bipartite, then $g\geq 4$, and our bound simplifies to $n \geq \frac{q\delta}{2}$. 
	Again, this bound is roughly double the existing bound of Theorem \ref{thm:intro_eppt_trianglefree} when $q \approx 2d$ and $q \approx 4r$.
	
	In the girth 5 case, Theorem \ref{thm:lower_bound} yields $n\geq \frac{q}{5}(\delta^2 + 1)$. 
	Although the author does not know of any lower bounds on the order of $(\delta+, 5)$-graphs in terms of diameter $d$, it is straightforward to get a rough diameter bound $n \geq \frac{d}{5}(\delta^2 + 1) + \mathcal{O}(\delta^2)$.
	To see this, consider a length $d$ geodesic $u_0, u_1, \dots, u_d$, and use the same counting argument as Lemma \ref{lem:low_bound_odd}.
	The $\mathcal{O}(\delta^2)$ term comes from accounting for extra / under-counted vertices at the start and end of the length $d$ geodesic.  
	Once more, the equator bound is about double this rough diameter bound for $\delta$ sufficiently large and $q \approx 2d$.
	
	For girth 6, our bound in \ref{thm:lower_bound} reduces to $n \geq \frac{q}{3}(\delta^2 - \delta + 1)$, which is again roughly double the diameter bound found in Theorem \ref{thm:intro_aa_girth6} \cite{alochukwu2021distances}.
	
	Using the same strategy as the proof of Lemma \ref{lem:low_bound_odd}, we bound the order of $C_4$-free graphs in Proposition \ref{prop:c4_lower_bound}.
	
	\begin{prop}
		\label{prop:c4_lower_bound}
		Let $G$ be a graph that does not contain any 4-cycles and has order $n$, minimum degree $\delta$ and equator $q > 15$. Then
		\begin{align*}
			n \geq \frac{q}{5}\left(\delta^2 - 2\left\lfloor \frac{\delta}{2} \right\rfloor + 1  \right).
		\end{align*}
	\end{prop}
	
	\begin{proof}
		Since $G$ does not contain any 4-cycles (induced or otherwise), $d_2(v) \geq \delta^2 - 2\lfloor \frac{\delta}{2} \rfloor + 1$ (as observed in \cite{erdHos1989radius}).
		To see this, note that no two neighbors of $v$ can have any other neighbor in common, but $N(v)$ may contain a matching of size at most $\lfloor \frac{d(v)}{2} \rfloor$. 
		The rest of the proof is essentially the same as the proof of Lemma \ref{lem:low_bound_odd} in the special case where $k=2$.
	\end{proof}
	
	Again, the bound in Proposition \ref{prop:c4_lower_bound} is, for graphs where with $q \approx 2d$ / $q \approx 4r$, roughly double the bounds in \cite{erdHos1989radius} (see Equation \ref{eqn:c_4_free_erdos}).
	It is interesting to note that if $v$ is a vertex in $C_4$-free graph, then $d_2(v) \geq \delta^2 - 2\lfloor \frac{\delta}{2} \rfloor + 1 = M(5, \delta) - 2\lfloor \frac{\delta}{2} \rfloor$.
	I.e., the minimum $2$-degree in a $C_4$-free graph is nearly as large as the minimum $2$-degree in a $\{C_3, C_4\}$-free graph.
	
	\section{Sharpness}
	\label{sec:sharpness}
	
	We show that the bound of Theorem \ref{thm:lower_bound} is sharp for at least some combinations of girth and minimum degree --- in particular, those for which there exists a Moore graph.
	We also construct, using cages, a larger class of graphs that nearly attain the bound of Theorem \ref{thm:lower_bound} for many combinations of girth and minimum degree.
	We conclude the section by showing that a well known construction is somewhat close to the bound for $C_4$-free graphs, and that we can get very close to the bound of Proposition \ref{prop:c4_lower_bound} when $\delta = 3$.
	
	\begin{thm}
		\label{thm:sharp_moore_parameters}
		Let $j\geq 3$ be an integer, and let $g, \delta$ be positive integers for which there exists a Moore graph of girth $g$ and minimum degree $\delta \geq 3$.
		Then there exists a $\delta$-regular graph $\mathcal{F}(\delta, g, q)$ with girth $g$, equator $q = jg$ and order $n = \frac{q}{g}M(\delta, g)$.
	\end{thm}
	
	\begin{proof}
		Let $H$ be a Moore graph of girth $g$ and minimum degree $\delta \geq 3$.
		Consider a length $g$ cycle $C$ of $H$, and an edge $uv$ of $C$.
		Note that in the graph $H-\{uv\}$, the vertices $u$ and $v$ are at the ends of a length $g-1$ geodesic $P$. 
		Create $j$ copies $F_1, F_2, \dots, F_j$ of the graph $H-\{uv\}$, and denote by $u_i$ and $v_i$ the copies of $u$ and $v$ in $F_i$, and denote by $P_i$ the copy of the geodesic $P$.
		Create $\mathcal{F}(\delta, g, q)$ by taking the disjoint union of the graphs $F_1 \cup \dots \cup F_j$, and adding all edges $v_1u_2, v_2u_3, \dots, v_{j_1}u_j, v_ju_1$. For examples, see Figure \ref{fig:sharp_example_3_4}.
		The graph $\mathcal{F}(\delta, g, q)$ is $\delta$-regular, has girth $g$, and the cycle formed from the $P_i$'s and the newly added edges is an isometric cycle of length $q = jg$.
		Since $\mathcal{F}(\delta, g, q)$ is formed from $j$ copies of $H$, it has order $jM(\delta, g) = \frac{q}{g}M(\delta, g)$.
	\end{proof}
	
	In particular, note that for all $\delta \geq 3$, the generic bound $n \geq \frac{q(\delta + 1)}{3}$ and the bipartite / triangle-free bound $n \geq \frac{q\delta}{2}$ are both sharp.
	The situation is more dire when we consider girth 5, where Theorem \ref{thm:sharp_moore_parameters} only yields sharp bounds for minimum degree $3, 7$ and (possibly) $57$ --- the degrees for which there (possibly) exists a girth 5 Moore graph \cite{hoffman1960moore}.
	Theorem \ref{thm:sharp_moore_parameters} is more effective for small even girths, as there are infinitely many Moore graphs with girths 6, 8 and 12 \cite{feit1964nonexistence}.
	
	\begin{figure}[h]
		\centering
		\begin{tikzpicture}[scale=0.6, inner sep=0.7mm]
	\begin{pgfonlayer}{nodelayer}
		\node [style=whitevertex] (0) at (-5.75, 1) {};
		\node [style=whitevertex] (1) at (-5.75, 0) {};
		\node [style=whitevertex] (2) at (-5.75, -1) {};
		\node [style=blackvertex] (3) at (-6.75, 0) {};
		\node [style=dgreyvertex] (4) at (-4.75, 0) {};
		\node [style=whitevertex] (5) at (-8.75, 1) {};
		\node [style=whitevertex] (6) at (-8.75, 0) {};
		\node [style=whitevertex] (7) at (-8.75, -1) {};
		\node [style=blackvertex] (8) at (-9.75, 0) {};
		\node [style=dgreyvertex] (9) at (-7.75, 0) {};
		\node [style=whitevertex] (10) at (-2.75, 1) {};
		\node [style=whitevertex] (11) at (-2.75, 0) {};
		\node [style=whitevertex] (12) at (-2.75, -1) {};
		\node [style=blackvertex] (13) at (-3.75, 0) {};
		\node [style=dgreyvertex] (14) at (-1.75, 0) {};
		\node [style=none] (15) at (-10.5, 0) {};
		\node [style=none] (16) at (-1, 0) {};
		\node [style=blackvertex] (17) at (1.75, 0) {};
		\node [style=whitevertex] (18) at (2.75, 1) {};
		\node [style=whitevertex] (19) at (2.75, -1) {};
		\node [style=whitevertex] (20) at (3.75, 1) {};
		\node [style=whitevertex] (21) at (3.75, -1) {};
		\node [style=dgreyvertex] (22) at (4.75, 0) {};
		\node [style=blackvertex] (23) at (5.75, 0) {};
		\node [style=whitevertex] (24) at (6.75, 1) {};
		\node [style=whitevertex] (25) at (6.75, -1) {};
		\node [style=whitevertex] (26) at (7.75, 1) {};
		\node [style=whitevertex] (27) at (7.75, -1) {};
		\node [style=dgreyvertex] (28) at (8.75, 0) {};
		\node [style=none] (29) at (1, 0) {};
		\node [style=none] (30) at (9.5, 0) {};
		\node [style=none] (31) at (-5.75, -2) {$(4,3)$};
		\node [style=none] (32) at (5.25, -2) {$(3,4)$};
	\end{pgfonlayer}
	\begin{pgfonlayer}{edgelayer}
		\draw (0) to (1);
		\draw (1) to (2);
		\draw [bend left] (0) to (2);
		\draw (3) to (0);
		\draw (3) to (1);
		\draw (3) to (2);
		\draw (0) to (4);
		\draw (1) to (4);
		\draw (2) to (4);
		\draw (5) to (6);
		\draw (6) to (7);
		\draw [bend left] (5) to (7);
		\draw (8) to (5);
		\draw (8) to (6);
		\draw (8) to (7);
		\draw (5) to (9);
		\draw (6) to (9);
		\draw (7) to (9);
		\draw (9) to (3);
		\draw (10) to (11);
		\draw (11) to (12);
		\draw [bend left] (10) to (12);
		\draw (13) to (10);
		\draw (13) to (11);
		\draw (13) to (12);
		\draw (10) to (14);
		\draw (11) to (14);
		\draw (12) to (14);
		\draw (4) to (13);
		\draw (15.center) to (8);
		\draw (14) to (16.center);
		\draw (18) to (17);
		\draw (17) to (19);
		\draw (18) to (21);
		\draw (20) to (19);
		\draw (18) to (20);
		\draw (19) to (21);
		\draw (20) to (22);
		\draw (22) to (21);
		\draw (24) to (23);
		\draw (23) to (25);
		\draw (24) to (27);
		\draw (26) to (25);
		\draw (24) to (26);
		\draw (25) to (27);
		\draw (26) to (28);
		\draw (28) to (27);
		\draw (22) to (23);
		\draw (28) to (30.center);
		\draw (29.center) to (17);
	\end{pgfonlayer}
\end{tikzpicture}
		\caption{On the left is a subgraph of the graph $\mathcal{F}(4, 3, q)$ constructed in the proof of Theorem \ref{thm:sharp_moore_parameters}.
		On the right is a subgraph of $\mathcal{F}(3, 4, q)$. In both diagrams, the vertices $u_i$ are black, and the vertices $v_i$ are grey.}
		\label{fig:sharp_example_3_4}
	\end{figure}
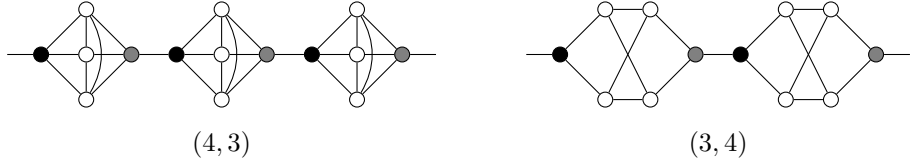 
	
	It is clear that construction in Theorem \ref{thm:sharp_moore_parameters} works not just when $H$ is a Moore graph, but when $H$ is any cage. 
	Thus we can use known cages to construct graphs with fixed equator, girth and minimum degree that have low order, yielding Remark \ref{rem:cage_examples}.
	Remember that we denote by $C(\delta, g)$ the order of a $(\delta, g)$-cage.
	
	\begin{rem}
		\label{rem:cage_examples}
		Let $j,\delta$ and $g$ all be positive integers greater than or equal to 3.
		There exists a $\delta$-regular graph $\mathcal{C}(\delta, g, q)$ with girth $g$, equator $q = jg$ and order $n = \frac{q}{g}C(\delta, g)$.
	\end{rem}
	
	Unfortunately, for some $g$ and $\delta$, our best bounds on $C(\delta, g)$ are substantially larger than $M(\delta, g)$ (for a detailed breakdown, see the surveys \cite{exoo2012dynamic, wong1982cages}). 
	However, for some combinations of interest there are small cages.
	The unique $(4,5)$-cage, also known as the Robertson graph, has order 19, which compares very favorably against $M(4,5) = 17$ \cite{robertson1964smallest}.
	The cage orders $C(5,5) = 30$ and $C(6,5) = 40$ are also not far off $M(5,5) = 26$ and $M(6,5) = 37$ \cite{wegner1973smallest, o1979smallest}.
	We summarise these comparisons in Table \ref{tab:moore_cage}.
	
	\begin{table}[h]
		\centering
		\begin{tabular}{c|cc}
			\hline
			$\delta$ & $M(\delta, 5)$ & $C(\delta, 5)$ \\ \hline
			3        & 10             & 10             \\
			4        & 17             & 19             \\
			5        & 26             & 30             \\
			6        & 37             & 40             \\ \hline
		\end{tabular}
		\caption{Table comparing the values of the Moore bound for girth 5 and the order of cages of girth 5.}
		\label{tab:moore_cage}
	\end{table}

	To study sharpness of the $C_4$-free bound in Proposition \ref{prop:c4_lower_bound}, we will need the Brown graphs, also known as Erd\"{o}s-R\'{e}nyi polarity graphs.
	To construct the \textbf{Brown graph} $B(t) = (V,E)$, consider a prime power $t$, and let $V$ be the elements of the projective plane $PG(2, t)$ over the finite field of order $t$ (i.e., equivalences classes $x = [x_1, x_2, x_3]$ where $x$ is equivalent to $y$ if $x = \lambda y$ for some scalar $\lambda$). 
	Let vertices $x$, $y$ be adjacent if their triples are orthogonal (i.e., $x\cdot y = 0$). 
	We recall properties of the graph $B(t)$ that were first used in \cite{brown1966graphs} and \cite{erdos1962problem}, and are discussed in detail in \cite{bachraty2015polarity}.
	
	\begin{lem}\textup{\cite{bachraty2015polarity, brown1966graphs, erdos1962problem}}
		\label{lem:brown_graph}
		Let $t$ be a prime power. The Brown graph $B(t)$ has the following properties:
		\begin{itemize}[noitemsep]
			\item $B(t)$ is $C_4$-free,
			\item $n = t^2 + t + 1$,
			\item Exactly $t+1$ vertices $x$ satisfy $x\cdot x = 0$, and these vertices have degree $t$,
			\item The remaining $t^2$ vertices have degree $t + 1$,
			\item If $x\cdot x = 0$, then $x$ is not in any triangle,
			\item There is exactly one length two path connecting any two vertices.
		\end{itemize}
	\end{lem}
	
	Using the Brown graphs, we can adapt a construction in \cite{erdHos1989radius} to show that the bound in Proposition \ref{prop:c4_lower_bound} is almost sharp.
	
	\begin{prop}
		\label{prop:sharp_c4_free}
		Let $j\geq 3$ and $\delta \geq 3$ be integers such that $\delta + 1$ is a power of some prime number.
		Then there exists a $C_4$-free graph $\mathcal{G}(\delta, q)$ with minimum degree $\delta$, equator $q = 5j$ and order $n = \frac{q}{5}(\delta^2 + 3\delta + 2)$.
	\end{prop}
	
	\begin{proof}
		Let $x$ be a vertex of $H = B(\delta + 1)$ such that $x\cdot x = 0$, and let $y$ and $z$ be two neighbors of $x$.
		Per Lemma \ref{lem:brown_graph}, the only $y-z$ path of length two or less is $y,x,z$. 
		Let $F$ be the graph formed from $H - \{x\}$ by removing all edges between a neighbor of $y$ and a neighbor of $z$ (i.e., remove the middle edge of all $y-z$ paths of length 3).
		In $F$, we now have $d(x,y) = 4$.
		Since $H$ is $C_4$-free, each neighbor of $y$ is adjacent to at most one neighbor of $z$ (and vice-versa).
		Thus each vertex of $F$ is incident with at most one edge of $H - F$, so the minimum degree of $F$ is at least $\delta$.
		Create $j$ copies $F_1, \dots, F_j$ of $F$, and denote by $y_i$ and $z_i$ the copies of $y$ and $z$ in $F_i$.
		The graph $\mathcal{G}(\delta, q)$ is formed from the disjoint union $F_1 \cup \dots \cup F_j$ by adding the edges $y_1x_2, y_2x_3, \dots, y_{j-1}x_j, y_jx_1$.
	\end{proof}
	
	When $\delta = 3$, we find a substantially smaller graph than the one derived from the Brown graph of Proposition \ref{prop:sharp_c4_free}.
	For $j\geq 3$, create a graph $\mathcal{H}(3, 6j)$ by taking $j$ copies $H_1 \cup \dots \cup H_j$ of the graph in Figure \ref{fig:delta3_noc4}, and adding all edges $v_1u_2, v_2u_3, \dots, v_qu_1$.
	
	\begin{figure}[h]
		\centering
		\begin{tikzpicture}[scale=0.9, inner sep=0.9mm]
	\begin{pgfonlayer}{nodelayer}
		\node [style=blackvertex, label={left:$u_i$}] (1) at (-4, 0) {};
		\node [style=whitevertex] (2) at (-2.5, 1) {};
		\node [style=whitevertex] (3) at (-2.5, -1) {};
		\node [style=whitevertex] (4) at (-1, 1) {};
		\node [style=whitevertex] (5) at (-1, -1) {};
		\node [style=whitevertex] (6) at (1, 1) {};
		\node [style=whitevertex] (7) at (0, 0) {};
		\node [style=whitevertex] (8) at (1, -1) {};
		\node [style=whitevertex] (9) at (2.5, 1) {};
		\node [style=whitevertex] (10) at (2.5, -1) {};
		\node [style=lgreyvertex, label={right:$v_i$}] (11) at (4, 0) {};
		\node [style=none] (12) at (0, -1.5) {$H_i$};
	\end{pgfonlayer}
	\begin{pgfonlayer}{edgelayer}
		\draw (1) to (2);
		\draw (1) to (3);
		\draw (3) to (2);
		\draw (2) to (4);
		\draw (3) to (5);
		\draw (4) to (7);
		\draw (7) to (5);
		\draw (4) to (6);
		\draw (6) to (7);
		\draw (7) to (8);
		\draw (8) to (5);
		\draw (6) to (9);
		\draw (8) to (10);
		\draw (9) to (10);
		\draw (10) to (11);
		\draw (11) to (9);
	\end{pgfonlayer}
\end{tikzpicture}
		\caption{Copies of the gadget $H_i$ above are used to build the $C_4$-free graph $\mathcal{H}(\delta, q)$.}
		\label{fig:delta3_noc4}
	\end{figure}
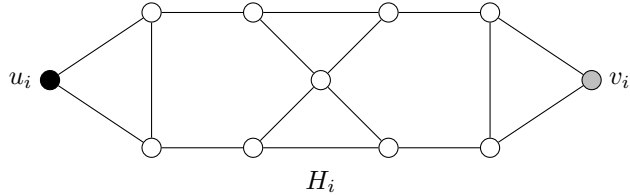 
	
	The $C_4$-free graph $\mathcal{H}(3, 6j)$ has equator $q = 6j$, minimum degree $3$ and order $n = 11j = \frac{11}{6}q = 1.8\overline{3}q$. 
	This is very nearly the bound in Proposition \ref{prop:c4_lower_bound}, which simplifies to $n \geq \frac{8}{5}q = 1.6q$ when $\delta = 3$.
	
	\section{Structure of Equatorial Graphs}
	\label{sec:equatorial_structure}
	
	We say a graph is an \textit{equatorial graph} if it has girth $g$, equator $q > 6k+3$, minimum degree $\delta$ and attains the order bound $n = \frac{q}{g}M(\delta, g)$ of Theorem \ref{thm:lower_bound}.
	These graphs are `equator-constrained' analogues of Moore graphs. 
	We devote this section to describing their structure.
	Like Moore graphs, all equatorial graphs are regular and every vertex of an equatorial graph lies on a maximum length isometric cycle. 
	Further, every equatorial graph comes equipped with a unique `canonical' partition into $q$ parts $L_0, \dots, L_{q-1}$ such that every equatorial cycle has exactly one vertex in each part, and different parts $L_i$ and $L_j$ are adjacent exactly when $|j-i| = 1$.
	As a consequence of the partition described above, an equatorial graph, considered as a graph with loops on every vertex, retracts onto any of its isometric $q$-cycles.
	
	We introduce a useful partition of their vertex sets.
	\begin{defn*}
		\label{def:partition}
		Let $C = u_0, u_1, \dots, u_{q-1}$ be an isometric $q$-cycle of an equatorial graph with girth $g$ and let $k = \lceil \frac{g}{2} \rceil - 1$. 
		Taking subscripts mod $q$, set
		\begin{align*}
			L_i = \dd_k(u_{i-k}) \cap \dd_k(u_{i+k}).
		\end{align*}
		We call $\{L_i : i\in [0,q-1]\}$ the \textit{partition of $G$ induced by $C$}.
	\end{defn*}
	Later, we show that this partition is independent of the choice of $C$.
	
	We show that all vertices are in the same numbers of disks $\dd_k(u_i)$, and each disk $\dd_k(u_i)$ contains the same number of vertices.
	
	\begin{lem}
		\label{lem:extremal_disk_counting}
		Suppose $G=(V,E)$ is an equatorial graph with equator $q$, girth $g$ and minimum degree $\delta$.
		Let $C = u_0, u_1, \dots, u_{q-1}$ be an isometric $q$-cycle of $G$, and denote $e_i = u_iu_{i+1}$ (subscripts mod $q$).
		\begin{itemize}[noitemsep]
			\item If $g$ is odd, then $d_k(u_i) = M(\delta, g)$ for every vertex $u_i$ of $C$.
			Further, for all vertices $v$ of $G$, the vertex $v$ is contained in exactly $g$ disks of the form $\dd_k(u_i)$.
			The indices $i$ such that $v\in \dd_k(u_i)$ form an interval $\{j-k, j-k+1, \dots, j+k\}$ (indices mod $q$).
			\item If $g$ is odd, then $d_k(e_i) = M(\delta, g)$ for every edge $e_i$ of $C$.
			Further, every vertex $v$ of $G$ is contained in exactly $g$ disks $\dd_k(e_i)$.
			These indices $i$ form an interval $\{j-k - 1, j-k, \dots, j+k\}$ (indices mod $q$).
		\end{itemize}
	\end{lem}
	
	\begin{proof}
		Let $g$ be odd. 
		From the proof of Lemma \ref{lem:low_bound_odd}, we have that each disk $\dd_k(u_i)$ has at least $M(\delta, g)$ vertices, and that each vertex $v$ is contained in at most $g$ such disks. 
		Suppose to the contrary that either $d_k(u_i) > M(\delta, g)$, or that some vertex $v$ is in fewer than $g$ such disks.
		Then the order of $G$ is greater than $\frac{q}{g}M(\delta, g)$ by the same counting argument as in the proof of Lemma \ref{lem:low_bound_odd}, a contradiction!
		Since $C$ is an isometric cycle of length at least $6k+3$, we have that $\dd_k(u_i) \cap \dd_k(u_j) = \emptyset$ whenever $|j-i| > 2k$.
		Therefore the indices $i$ such that $v\in \dd_k(u_i)$ form an interval of $2k+1 = g$ consecutive integers.
		The proof for $g$ even is similar (use Lemma \ref{lem:low_bound_even} instead of Lemma \ref{lem:low_bound_odd}).
	\end{proof}

	The next series of lemmas describe the partitions of an equatorial graph induced by some isometric $q$-cycle.
	The first of these lemmas shows the sets $L_i$ have a very convenient representation in graphs of even girth.
	
	\begin{lem}
		\label{lem:ti_is_li}
		Let $G = (V,E)$ be an equatorial graph with even girth $g = 2k+2$. 
		Suppose that $C$ is some isometric $q$-cycle, and $L_i$ is defined as in Definition \ref{def:partition}.
		Then 
		\[
			L_i = \dd_k(e_{i-k-1}) \cap \dd_k(e_{i+k}).
		\]
	\end{lem}
	
	\begin{proof}
		It's clear that $L_i \subseteq \dd_k(e_{i-k-1}) \cap \dd_k(e_{i+k})$.
		Suppose that $v \in \dd_k(e_{i-k-1}) \cap \dd_k(e_{i+k})$. 
		We show that $v\in L_i$.
		Since $v \in \dd_k(e_{i-k-1})$, we have that $d(v, u_{i-k-1}) = k$ or $d(v, u_{i-k}) = k$. 
		If $d(v, u_{i-k-1}) = k$, then by the triangle inequality we get
		\[
			d(u_{i-k-1}, e_{i+k}) \leq d(u_{i-k-1}, v) + d(v, e_{i+k}) \leq 2k.
		\]
		However, this contradicts that $C$ is an isometric cycle.
		Therefore $d(v, u_{i-k}) = k$.
		The argument that $d(v, u_{i+k}) = k$ is similar, and we conclude that $v \in \dd_k(u_{i-k}) \cap \dd_k(u_{i+k}) = L_i$.
	\end{proof}
	
	Lemma \ref{lem:li_partition} confirms that $\{L_0, \dots, L_{q-1}\}$ is a partition of the equatorial graph's vertex set.
	
	\begin{lem}
		\label{lem:li_partition}
		Let $G=(V,E)$ be an equatorial graph and let $C=u_0, \dots, u_{q-1}$ be an isometric $q$-cycle of $G$.
		Then the sets $L_i$ partition $V$, and $u_i \in L_i$.
	\end{lem}
	
	\begin{proof}
		That $u_i \in L_i$ is immediate, since $C$ is isometric.
		For each vertex $v\in V$, we show that there is a unique set $L_i$ containing $v$.
		
		Let $g$ be odd.
		By Lemma \ref{lem:extremal_disk_counting}, there is an interval $\{j-k, \dots, j+k\}$ such that $v\in \dd_i$ for all $i \in \{j-k, \dots, j+k\}$ (integers mod $q$).
		Thus $v \in L_j$.
		Suppose now to the contrary that $v$ is in two different sets $L_i$ and $L_j$.
		Since $v$ is in both $L_i$ and $L_j$, we have that $d(v,w) \leq k$ for all $w \in \{u_{i-k}, u_{i+k}, u_{j-k}, u_{j+k}\}$. 
		Thus, for all $s \in \{u_{i-k}, u_{i+k}\}$ and $t\in \{u_{j-k}, u_{j+k}\}$, we have $d(s,t) \leq 2k$ by the triangle inequality.
		There are two cases to consider.
		
		\textit{Case 1:} $d(u_i, u_j) \leq 2k$. 
		Since $C$ is an isometric cycle of length $q>6k+3$, either $d(u_{i+k}, u_{j-k}) > 2k$ or $d(u_{i-k}, u_{j+k}) > 2k$, contradicting the previous statement that $d(s,t) \leq 2k$.

		\textit{Case 2:} $d(u_i, u_j) > 2k$.
		In this case, $d(u_{i+k}, u_{j+k}) > 2k$, a contradiction.
		
		Now suppose $g$ is even.
		By Lemma \ref{lem:ti_is_li}, $L_i = \dd_k(e_{i-k-1}) \cap \dd_k(e_{i+k})$.
		As in the odd case, we have that $v$ is in some set $L_i$ by an application of Lemma \ref{lem:extremal_disk_counting}.
		The argument that $L_i$ is unique is the same as in the case where $g$ is odd.
	\end{proof}
	
	\begin{lem}
		\label{lem:disk_in_interval}
		Let $C = u_0, \dots, u_{q-1}$ be an isometric $q$-cycle of an equatorial graph of girth $g$, and let $i\in [0, q-1]$.
		\begin{itemize}[noitemsep]
			\item If $g$ is odd, then $\dd_k(u_i)  = L_{i-k} \cup \dots \cup L_{i+k}$.
			\item If $g$ is even, then $\dd_k(e_i) = L_{i-k} \cup \dots \cup L_{i+k+1}$. 
		\end{itemize}
	\end{lem}
	
	\begin{proof}
		Name the equatorial graph $G=(V,E)$.
		First suppose that $g$ is odd.
		For convenience, denote $\Lambda_i = L_{i-k} \cup \dots \cup L_{i+k}$.
		We begin by showing that $\dd_k(u_i) \subseteq \Lambda_i$.
		Assume to the contrary that there exists a vertex $v$ of $\dd_k(u_i)$ that is not in in $\Lambda_i$.
		As $\{L_0, \dots, L_{q-1}\}$ is a partition of $V$, the vertex $v$ is in $L_j$ for some $j \notin [i-k, \dots, i+k]$.
		We thus have that $d(v,w) \leq k$ for all $w\in \{u_i, u_{j-k}, u_{j+k}\}$.
		By the triangle inequality we get both $d(u_i, u_{j-k}) \leq 2k$ and $d(u_i, u_{j+k}) \leq 2k$, contradicting the fact that $C$ is an isometric cycle of length $q > 6k+3$.
		So $\dd_k(u_i) \subseteq \Lambda_i$, and thus $|\Lambda_i| \geq |\dd_k(u_i)| = M(\delta, g)$ (Lemma \ref{lem:extremal_disk_counting}).
		Conversely, suppose that $|\Lambda_i| > |\dd_k(u_i)|$.
		Since $\{L_0, L_1, \dots, L_{q-1}\}$ is a partition of $V$ and each set $L_s$ is a subset of exactly $2k+1 = g$ of the sets $\Lambda_t$, we can bound the order of $G$ by
		\begin{align*}
			n \geq \frac{|\Lambda_0| + \dots + |\Lambda_i| + \dots + |\Lambda_{q-1}|}{g} \geq \frac{(q-1)M(\delta, g) + |\Lambda_i|}{g} > \frac{qM(\delta, g)}{g}.
		\end{align*}
		This contradicts that $G$ is an equatorial graph, so $\dd_k(u_i) = \Lambda_i$.
		
		Now suppose $g$ is even, and let $\Gamma_i = L_{i-k} \cup \dots \cup L_{i+k+1}$.
		Note that $q > 6k+3$.
		Assume, contrary to the statement $\dd_k(e_i) \subseteq \Gamma_i$, that there is a vertex $v$ of $\dd_k(e_i)$ that is in $L_j$ for some $j \notin [i-k, i+k+1]$.
		Thus $d(v, u_{j-k}) \leq k$, $d(v, u_{j+k}) \leq k$ and $d(v, e_i) \leq k$. 
		By the triangle inequality, both $d(u_{j-k}, e_i) \leq 2k$ and $d(u_{j+k}, e_i) \leq 2k$, contradicting that $C$ is isometric.
		So $\dd_k(e_i) \subseteq \Gamma_i$ and thus $|\Gamma_i| \geq |\dd_k(e_i)| = M(\delta, g)$.
		The rest of the argument that $\dd_k(e_i) = \Gamma_i$ is the same as in the odd case --- but take care to note that each $L_i$ is a subset of $2k+2 = g$ of the sets $\Gamma_i$.
	\end{proof}
	
	\begin{lem}
		\label{lem:edges_between_li}
		Let $G$ be an equatorial graph with equator $q$, $C$ an isometric $q$-cycle of $G$, and $\{L_i : i \in [0,q-1] \}$ the partition of $G$ induced by $C$.
		A vertex of $L_i$ is adjacent to a vertex of $L_j$ only if $|j-i| \leq 1$ (difference taken mod $q$).
	\end{lem}
	
	\begin{proof}
		Let $g$ be the girth of $G$, let $k = \lceil \frac{g}{2} \rceil - 1$ and remember $q > 6k+3$.
		Assume to the contrary that there is an edge from some vertex $x \in L_i$ to $y \in L_j$ where $|i-j| \geq 2$ (difference taken mod $q$).
		Note that $d(x, u_{i-k}) = k = d(x, u_{i+k})$, and similarly $d(y, u_{j-k}) = k = d(y, u_{j+k})$.
		By the triangle inequality, for all $s \in \{u_{i-k}, u_{i+k}\}$ and $t \in \{u_{j-k}, u_{j+k}\}$, we have that $d(s,t) \leq 2k+1$.
		There are two cases to consider: when $d(u_i, u_j) \leq 2k+1$, and when $d(u_i, u_j) > 2k+1$.
		We obtain a contradiction by a similar argument to that used in Lemma \ref{lem:li_partition}.
	\end{proof}
	
	\begin{prop}
		\label{prop:all_in_isocycles}
		Every vertex of an equatorial graph is in some isometric $q$-cycle.
	\end{prop}
	
	\begin{proof}
		Let $C = u_0, \dots, u_{q-1}$ be an isometric $q$-cycle of the equatorial graph $G=(V,E)$, and let $v$ be an arbitrary vertex of $G$.
		Since the sets $L_i$ partition $V$ (Lemma \ref{lem:li_partition}), the vertex $v$ is in some set $L_i = \dd_k(u_{i-k}) \cap \dd_k(u_{i+k})$.
		Since $C$ is isometric, we must have that $d(u_{i-k}, v) = k = d(v, u_{i+k})$.
		Thus $v$ is the middle vertex of a length $2k$ geodesic $P = u_{i-k}, \dots, v, \dots, u_{i+k}$.
		Create a new cycle $C' = (C - \{u_{i-k+1}, \dots, u_{i+k-1}\}) \cup P$ from $C$ by removing the shorter path of $C$ between $u_{i-k}$ and $u_{i+k}$, and replacing it with the path $P$.
		It is clear that $C'$ has length $q$ and contains $v$.
		We now show that $C'$ is isometric.
		Relabel the vertices of $C' = w_0, w_1, \dots, w_{q-1}$ so that $w_0 = u_0$, $w_{i-k} = u_{i-k}$, $w_i = v$ and so on.
		Thus, for all $j \notin [i-k+1, i+k-1]$, we have $w_j \in L_j$ by Lemma \ref{lem:li_partition}.
		Since $P$ is a path of length $2k$ from a vertex of $L_{i-k}$ to $L_{i+k}$, we apply Lemma \ref{lem:edges_between_li} to deduce that $w_j \in L_j$ for all $j \in [i-k+1, i+k-1]$.
		So $w_j \in L_j$ for all $j \in [0, q-1]$.
		By Lemma \ref{lem:edges_between_li}, any path from some vertex $w_s \in C'$ to $w_t \in C'$ has length at least $|s-t|$ (difference mod $q$), so $C'$ is isometric.
	\end{proof} 
	
	\begin{prop}
		\label{prop:equatorial_regular}
		Every equatorial graph is regular.
	\end{prop}
	
	\begin{proof}
		Suppose to the contrary that an equatorial graph $G$ has minimum degree $\delta$, girth $g$ and contains a vertex $v$ with $d(v) > \delta$.
		By Proposition \ref{prop:all_in_isocycles}, there is some isometric $q$-cycle $C = u_0, \dots, u_i = v, u_{q-1}$ containing $v$ as its $i$th vertex.
		First suppose $g$ is odd.
		Since $G$ has girth $g$ and minimum degree $\delta$, we have
		\begin{align*}
			|\dd_k(v)| &\geq |\{v\}| + |N_1(v)| + \dots + |N_k(v)|\\
				&\geq 1 + d(v) + d(v)(\delta - 1) + \dots + d(v)(\delta - 1)^{k-1}\\ 
				&> M(\delta, g).
		\end{align*}
		However $|\dd_k(v)| = M(\delta, g)$ by Lemma \ref{lem:extremal_disk_counting} --- a contradiction!
		Now suppose $g$ is even, and let $e_i = vu_{i+1}$.
		As $G$ has girth $g$ and minimum degree $\delta$, we get
		\begin{align*}
			|\dd_k(vu_{i+1})| &\geq |\{vu_{i+1}\}| + |N_1(vu_{i+1})| + \dots + |N_k(vu_{i+1})|\\
			&\geq 2 + [(d(v) - 1) + (\delta - 1)] + \sum_{i=1}^k \left[(d(v) - 1)(\delta - 1)^{i-1} + (\delta - 1)^i\right]\\
			&> M(\delta, g).
		\end{align*}
		As in the odd case, this contradicts Lemma \ref{lem:extremal_disk_counting}.
	\end{proof}
	
	The next results establish that the partition $\{L_i : i\in [0,q-1]\}$ of an equatorial graph is unique. 
	I.e., every isometric $q$-cycle induces the same partition.
	
	\begin{lem}
		\label{lem:one_vertex_per_li}
		Let $G$ be an equatorial graph, $C$ an isometric $q$-cycle of $G$ and $\{L_0, \dots, L_{q-1}\}$ the partition of $G$ induced by $C$. 
		If $D$ is any isometric $q$-cycle of $G$, then $|V(D) \cap L_i| = 1$ for all $i\in [0,q-1]$.
	\end{lem}
	
	\begin{proof}
		Let $G$ be an equatorial graph of girth $g$ and equator $q$, $C = u_0, \dots, u_{q-1}$ an isometric $q$-cycle and $\{L_i : i\in  [0,q-1]\}$ the partition induced by $C$.
		Further, suppose that $D = v_0, \dots, v_{q-1}$ is another isometric $q$-cycle of $G$.
		Assume to the contrary that some part $L_j$ of the partition induced by $C$ does not contain any vertex of $D$.
		Since every edge of $G$ either lies within a part $L_i$, or between vertices of two consecutive parts $L_i$ and $L_{i+1}$ (Lemma \ref{lem:edges_between_li}), the cycle $D$ spans at most $\lfloor \frac{q}{2} \rfloor + 1$ consecutive parts $L_i$.
		Up to relabeling of the vertices of $C$ (and thus renaming parts $L_i$), we have that 
		\[
			V(D) \subseteq L_0 \cup L_1 \cup \dots \cup L_t, \text{ with } t \leq \left\lfloor \frac{q}{2} \right\rfloor.
		\]
		
		We prove that $D$ is not isometric by showing that $d_G(v_s, v_t) \leq 2k$ for some pair $v_s$ and $v_t$ of vertices such that $d_D(v_s, v_t) \geq \lceil \frac{q}{2} \rceil - 1 > 2k$. 
		Suppose $x$ and $y$ are any two vertices of $G$.
		There exist parts $L_i$ and $L_j$ of the partition induced by $C$ such that $x \in L_i$ and $y\in L_j$.
		Define the function $\varphi : V\times V \to \mathbb{Z}$ by $\varphi(x, y) = i - j$ (we do \textit{not} consider this difference mod $q$). 
		Note that $\varphi(y,x) = j - i = - \varphi(x,y)$, and that $\varphi(x,y) = 0$ if and only if $x$ and $y$ are in the same set $L_i$.
		By Lemma \ref{lem:edges_between_li}, every neighbor $x'$ of $x$ is in $L_{i-1} \cup L_i \cup L_{i+1}$. 
		Thus if $x$ is adjacent to $x'$ then $|\varphi(x,y) - \varphi(x', y)| \in \{-1,0,1\}$.
		Similarly, if $y'$ is adjacent to $y$ then $|\varphi(x,y) - \varphi(x, y')| \in \{-1,0,1\}$.
		Consider the following sequence of values of $\varphi$:
		\[
			\varphi(v_0, v_{\lceil \frac{q}{2} \rceil}), \varphi(v_1, v_{\lceil \frac{q}{2} \rceil}), \varphi(v_1, v_{\lceil \frac{q}{2} \rceil + 1}), \varphi(v_2, v_{\lceil \frac{q}{2} \rceil + 1}), \varphi(v_2, v_{\lceil \frac{q}{2} \rceil + 2}), \dots, \varphi(v_{\lceil \frac{q}{2} \rceil}, v_0).
		\]
		Consecutive terms of the sequence are either the same, or differ by exactly 1, and $\varphi(v_0, v_{\lceil \frac{q}{2} \rceil}) = -\varphi(v_{\lceil \frac{q}{2} \rceil}, v_0)$.
		Therefore there exists some some term $\varphi(v_s, v_t) = 0$, so $v_s$ and $v_t$ are both in the same part $L_i$ of the partition induced by $C$.
		The sequence is constructed so that either $d_D(v_s, v_t) = \lceil \frac{q}{2} \rceil - 1$ or $d_D(v_s, v_t) = \lceil \frac{q}{2} \rceil$.
		However, each vertex of $L_i$ is distance $k$ from $u_{i-k}$, so
		\[
			d_G(v_s, v_t) \leq d_G(v_s, u_{i-k}) + d_G(u_{i-k}, v_t) \leq 2k < \left\lceil \frac{q}{2} \right\rceil - 1.
		\] 
		This contradicts the fact that $D$ is isometric.
	\end{proof}
	
	\begin{prop}
		\label{prop:unique_partition}
		Let $G$ be an equatorial graph. 
		Every isometric $q$-cycle of $G$ induces the same partition.
	\end{prop}
	
	\begin{proof}
		Let $G$ be an equatorial graph of girth $g$, equator $q$ and set $k = \lceil \frac{g}{2} \rceil - 1$.
		Suppose $C = u_0, \dots, u_{q-1}$ and $D = v_0, \dots, v_{q-1}$ are two isometric $q$-cycles of $G$.
		Let $\mathbb{L} = \{L_i : i \in [0,q-1] \}$ be the partition induced by $C$ and $\mathbb{M} = \{M_i : i \in [0,q-1] \}$ be the partition induced by $D$.
		By Lemma \ref{lem:one_vertex_per_li}, we have that each $L_i$ contains a unique vertex of $D$. 
		By Lemma \ref{lem:edges_between_li}, two vertices of $D$ are adjacent exactly when they are in consecutive parts $L_i$ and $L_{i+1}$.
		Therefore, up to relabeling of the vertices of $D$ (and thus also renaming the parts of $\mathbb{M}$), we have that $v_i \in L_i$ for all $i\in [0, q-1]$.
		Also note that for all $i\in [0, q-1]$, both $u_i \in L_i$ and $v_i \in M_i$ by Lemma \ref{lem:li_partition}.
		To prove that $\mathbb{L} = \mathbb{M}$, it suffices to prove that if a vertex $w$ is in $L_i \cap M_j$ then $i = j$.
		
		Let $w$ be any vertex of $G$, and suppose $w$ is in $L_i \cap M_j$.
		Since $w \in L_i$, we have $d(w, u_{i-k}) \leq k$ and $d(w, u_{i+k}) \leq k$.
		Since $w \in M_j$, we get $d(w, v_{j-k}) \leq k$ and $d(w, u_{j+k}) \leq k$.
		By the triangle inequality, we have that for all $s \in \{u_{i-k}, u_{i+k}\}$ and $t \in \{v_{j-k}, v_{j+k}\}$, $d(s,t) \leq 2k$.
		Since $d(u_{i-k}, v_{j-k}) \leq 2k$ (and also $d(u_{i+k}, v_{j+k}) \leq 2k$), we have $j \in [i-2k, i+2k]$ (interval and indices mod $q$).
		As $d(u_{i+k}, v_{j-k}) \leq 2k$, we further restrict $j \in [i, i+2k]$.
		Similarly $d(u_{i-k}, v_{j+k}) \leq 2k$ implies $j \in [i-2k, i]$.
		Thus $j = i$, completing the proof.
	\end{proof}
	
	Let $u$ be a vertex of an equatorial graph, and suppose it is in part $L_i$ of the partition induced by isometric $q$-cycles.
	Since $u$ belongs to some isometric $q$-cycle, and the neighbors of $u$ on this cycle are in $L_{i-1}$ and $L_{i+1}$, we have Corollary \ref{cor:neighbors_in_li}.
	
	\begin{cor}
		\label{cor:neighbors_in_li}
		Let $u$ be a vertex of an equatorial graph in part $L_i$ of the partition induced by isometric cycles.
		Then $u$ has a neighbor in $L_{i-1}$ and a neighbor in $L_{i+1}$.
	\end{cor}
	
	Before tying up all of our structural results, we need one further lemma constraining the sizes of the parts $L_i$.
	
	\begin{lem}
		\label{lem:periodic_li}
		Let $G$ be an equatorial graph with girth $g$, and suppose $\{L_0, \dots, L_{q-1}\}$ is the partition of $G$ induced by isometric cycles.
		Then $|L_i| = |L_{i+g}|$ for all $i \in [0, q-1]$ (subscripts mod $q$).
	\end{lem}
	
	\begin{proof}
		By Lemmas \ref{lem:extremal_disk_counting} and \ref{lem:disk_in_interval}, we have $|L_{i}| + \dots + |L_{i+g-1}| = |\dd_k(u_{i+k})| = M(\delta, g)$ if $g$ is odd, and $|L_{i}| + \dots + |L_{i+g-1}| = |\dd_k(e_{i+k})| = M(\delta, g)$ if $g$ is even.
		Similarly, $|L_{i+1}| + \dots + |L_{i+g}| = M(\delta, g)$. 
		Therefore $|L_{i}| + \dots + |L_{i+g-1}| = |L_{i+1}| + \dots + |L_{i+g}|$, and so $|L_i| = |L_g|$.
	\end{proof}
	
	Theorem \ref{thm:equatorial_structure} gives our final description of any arbitrary equatorial graph, and is straightforward to prove from previous results of this section.
	
	\begin{thm}
		\label{thm:equatorial_structure}
		Let $G=(V,E)$ be an equatorial $(\delta+, g, q)$-graph.
		Then $G$ is $\delta$-regular and every vertex lies on an isometric $q$-cycle.
		The vertex set $V$ admits a unique partition ${V = L_0 \cup \dots \cup L_{q-1}}$ that satisfies the properties below. 
		For this list, let $u$ be any vertex in some part $L_i$ of the partition, $k = \lceil \frac{g}{2} \rceil - 1$ and consider all subscripts mod $q$.
		\begin{itemize}[noitemsep]
			\item The vertex $u$ has neighbors in $L_{i-1}$ and $L_{i+1}$. 
			\item $\dd_1(u) \subseteq L_{i-1} \cup L_i \cup L_{i+1}$.
			\item Every isometric $q$-cycle of $G$ has exactly one vertex in each part $L_i$.
			\item For any vertices $v \in L_{i-k}$ and $w \in L_{i+k}$, we have $L_i = \dd_k(v) \cap \dd_k(w)$. 
			\item $|L_j| = |L_{j+g}|$ for all $j$.
			\item $|L_j| + \dots + |L_{j+g-1}| = M(\delta, g)$ for all $j$.
		\end{itemize}
	\end{thm}
	
	Let $C_q^\circ$ be the cycle of length $q$ with loops on every vertex.
	Theorem \ref{thm:equatorial_structure} yields the following remark concerning homomorphisms and retracts from an equatorial graph.
	
	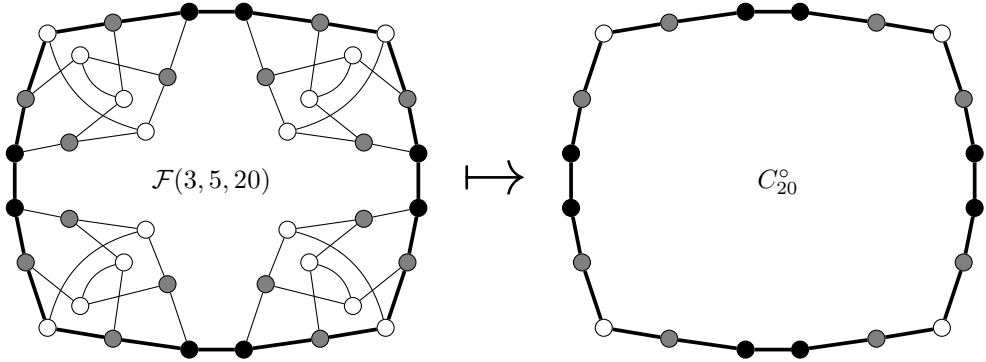
\begin{figure}[h]
		\centering
		\begin{tikzpicture}[scale=0.58, inner sep= 0.8mm]
	\begin{pgfonlayer}{nodelayer}
		\node [style=blackvertex] (0) at (-9.5, 2.5) {};
		\node [style=dgreyvertex] (2) at (-9.25, 3.75) {};
		\node [style=dgreyvertex] (3) at (-8.25, 2.75) {};
		\node [style=whitevertex] (4) at (-8.75, 5.25) {};
		\node [style=whitevertex] (5) at (-8, 4.75) {};
		\node [style=whitevertex] (6) at (-7, 3.75) {};
		\node [style=whitevertex] (7) at (-6.5, 3) {};
		\node [style=dgreyvertex] (8) at (-7.25, 5.5) {};
		\node [style=dgreyvertex] (9) at (-6, 4.25) {};
		\node [style=blackvertex] (10) at (-5.5, 5.75) {};
		\node [style=blackvertex] (11) at (-0.25, 2.5) {};
		\node [style=dgreyvertex] (12) at (-0.5, 3.75) {};
		\node [style=dgreyvertex] (13) at (-1.5, 2.75) {};
		\node [style=whitevertex] (14) at (-1, 5.25) {};
		\node [style=whitevertex] (15) at (-1.75, 4.75) {};
		\node [style=whitevertex] (16) at (-2.75, 3.75) {};
		\node [style=whitevertex] (17) at (-3.25, 3) {};
		\node [style=dgreyvertex] (18) at (-2.5, 5.5) {};
		\node [style=dgreyvertex] (19) at (-3.75, 4.25) {};
		\node [style=blackvertex] (20) at (-4.25, 5.75) {};
		\node [style=blackvertex] (21) at (-9.5, 1.25) {};
		\node [style=dgreyvertex] (22) at (-9.25, 0) {};
		\node [style=dgreyvertex] (23) at (-8.25, 1) {};
		\node [style=whitevertex] (24) at (-8.75, -1.5) {};
		\node [style=whitevertex] (25) at (-8, -1) {};
		\node [style=whitevertex] (26) at (-7, 0) {};
		\node [style=whitevertex] (27) at (-6.5, 0.75) {};
		\node [style=dgreyvertex] (28) at (-7.25, -1.75) {};
		\node [style=dgreyvertex] (29) at (-6, -0.5) {};
		\node [style=blackvertex] (30) at (-5.5, -2) {};
		\node [style=blackvertex] (31) at (-0.25, 1.25) {};
		\node [style=dgreyvertex] (32) at (-0.5, 0) {};
		\node [style=dgreyvertex] (33) at (-1.5, 1) {};
		\node [style=whitevertex] (34) at (-1, -1.5) {};
		\node [style=whitevertex] (35) at (-1.75, -1) {};
		\node [style=whitevertex] (36) at (-2.75, 0) {};
		\node [style=whitevertex] (37) at (-3.25, 0.75) {};
		\node [style=dgreyvertex] (38) at (-2.5, -1.75) {};
		\node [style=dgreyvertex] (39) at (-3.75, -0.5) {};
		\node [style=blackvertex] (40) at (-4.25, -2) {};
		\node [style=none] (41) at (1.5, 1.85) {\Huge $\mapsto$};
		\node [style=blackvertex] (42) at (3.25, 2.5) {};
		\node [style=dgreyvertex] (43) at (3.5, 3.75) {};
		\node [style=whitevertex] (45) at (4, 5.25) {};
		\node [style=dgreyvertex] (49) at (5.5, 5.5) {};
		\node [style=blackvertex] (51) at (7.25, 5.75) {};
		\node [style=blackvertex] (52) at (12.5, 2.5) {};
		\node [style=dgreyvertex] (53) at (12.25, 3.75) {};
		\node [style=whitevertex] (55) at (11.75, 5.25) {};
		\node [style=dgreyvertex] (59) at (10.25, 5.5) {};
		\node [style=blackvertex] (61) at (8.5, 5.75) {};
		\node [style=blackvertex] (62) at (3.25, 1.25) {};
		\node [style=dgreyvertex] (63) at (3.5, 0) {};
		\node [style=whitevertex] (65) at (4, -1.5) {};
		\node [style=dgreyvertex] (69) at (5.5, -1.75) {};
		\node [style=blackvertex] (71) at (7.25, -2) {};
		\node [style=blackvertex] (72) at (12.5, 1.25) {};
		\node [style=dgreyvertex] (73) at (12.25, 0) {};
		\node [style=whitevertex] (75) at (11.75, -1.5) {};
		\node [style=dgreyvertex] (79) at (10.25, -1.75) {};
		\node [style=blackvertex] (81) at (8.5, -2) {};
		\node [style=none] (82) at (8, 1.85) {$C_{20}^\circ$};
		\node [style=none] (83) at (-5, 1.85) {$\mathcal{F}(3,5,20)$};
	\end{pgfonlayer}
	\begin{pgfonlayer}{edgelayer}
		\draw [style = thick] (2) to (0);
		\draw (0) to (3);
		\draw [style=thick] (4) to (2);
		\draw (2) to (5);
		\draw (3) to (6);
		\draw (3) to (7);
		\draw (5) to (9);
		\draw (9) to (7);
		\draw (8) to (6);
		\draw [style=thick] (4) to (8);
		\draw [style=thick] (8) to (10);
		\draw (10) to (9);
		\draw [bend right] (5) to (6);
		\draw [bend right] (4) to (7);
		\draw [style=thick] (12) to (11);
		\draw (11) to (13);
		\draw [style=thick] (14) to (12);
		\draw (12) to (15);
		\draw (13) to (16);
		\draw (13) to (17);
		\draw (15) to (19);
		\draw (19) to (17);
		\draw (18) to (16);
		\draw [style=thick] (14) to (18);
		\draw [style=thick] (18) to (20);
		\draw (20) to (19);
		\draw [bend left] (15) to (16);
		\draw [bend left] (14) to (17);
		\draw [style=thick] (10) to (20);
		\draw [style=thick] (22) to (21);
		\draw (21) to (23);
		\draw [style=thick] (24) to (22);
		\draw (22) to (25);
		\draw (23) to (26);
		\draw (23) to (27);
		\draw (25) to (29);
		\draw (29) to (27);
		\draw (28) to (26);
		\draw [style=thick] (24) to (28);
		\draw [style=thick] (28) to (30);
		\draw (30) to (29);
		\draw [bend left] (25) to (26);
		\draw [bend left] (24) to (27);
		\draw [style=thick] (32) to (31);
		\draw (31) to (33);
		\draw [style=thick] (34) to (32);
		\draw (32) to (35);
		\draw (33) to (36);
		\draw (33) to (37);
		\draw (35) to (39);
		\draw (39) to (37);
		\draw (38) to (36);
		\draw [style=thick] (34) to (38);
		\draw [style=thick] (38) to (40);
		\draw (40) to (39);
		\draw [bend right] (35) to (36);
		\draw [bend right] (34) to (37);
		\draw [style=thick] (30) to (40);
		\draw [style=thick] (0) to (21);
		\draw [style=thick] (11) to (31);
		\draw [style=thick] (43) to (42);
		\draw [style=thick] (45) to (43);
		\draw [style=thick] (45) to (49);
		\draw [style=thick] (49) to (51);
		\draw [style=thick] (53) to (52);
		\draw [style=thick] (55) to (53);
		\draw [style=thick] (55) to (59);
		\draw [style=thick] (59) to (61);
		\draw [style=thick] (51) to (61);
		\draw [style=thick] (63) to (62);
		\draw [style=thick] (65) to (63);
		\draw [style=thick] (65) to (69);
		\draw [style=thick] (69) to (71);
		\draw [style=thick] (73) to (72);
		\draw [style=thick] (75) to (73);
		\draw [style=thick] (75) to (79);
		\draw [style=thick] (79) to (81);
		\draw [style=thick] (71) to (81);
		\draw [style=thick] (42) to (62);
		\draw [style=thick] (52) to (72);
	\end{pgfonlayer}
\end{tikzpicture}
		\caption{The equatorial graph $\mathcal{F}(3,5,20)$, considered as a reflexive graph with loops on every vertex, retracts onto the bolded isometric cycle. If two vertices are in the same part $L_i$ of the partition induced by isometric cycles, then they are drawn the same color.}
		\label{fig:homomorphism}
	\end{figure} 
	\begin{rem}
		\label{rem:equatorial_retract}
		Let $G$ be an equatorial graph with equator $q$.
		Then there exists a homomorphism $f : G \to C_q^\circ$ that is surjective on non-loop edges.
		If $C$ is any isometric $q$-cycle of $G$, then the vertices of $C_q^\circ$ can be relabeled to make $f$ a retraction of $G$ onto $C$.
	\end{rem}
	
	\section{Existence of Equatorial Graphs}
	\label{sec:extremal_parameters}
	
	In this section we show that if there are any equatorial graphs with minimum degree $\delta$ and girth $g$, then there are infinitely many. 
	This is a stark departure from the situation for Moore graphs, as there are only finitely many with odd girth.
	We also give a condition under which the existence of an equatorial $(\delta, g)$-graph implies the existence of a Moore $(\delta, g)$-graph.
	
	\begin{prop}
		\label{prop:infinitely_many_q}
		Suppose there exists an equatorial graph with minimum degree $\delta$, girth $g$ and equator $q$.
		Then for all integers $j\geq 2$, there exists an equatorial graph with the same minimum degree and girth having equator $jq$.
	\end{prop}
	
	\begin{proof}
		Let $G=(V,E)$ be an equatorial graph of minimum degree $\delta$, girth $g$ and equator $q$ and let $\{L_i : i\in [0,q-1]\}$ be the partition of $G$ induced by its isometric $q$-cycles.
		Form a graph $H$ from $G$ by removing all edges of the form $uv$, where $u\in L_{q-1}$ and $v\in L_0$.
		Now consider $j$ copies $H_0, \dots, H_{j-1}$ of $H$. 
		If $u$ and $v$ are vertices of $H$, denote the copies of $u$ and $v$ in $H_i$ by $u^i$ and $v^i$, respectively.
		Further, denote $L_{iq + t}'$ the copy of $L_t$ in $H_i$. 
		
		Construct a graph $\mathcal{J}$ from the disjoint union $H_0 \cup \dots \cup H_{j-1}$ by adding all edges of the form $u^iv^{i+1}$ where $u \in L_{q-1}$, $v \in L_0$ and $uv \in E$ (superscripts mod $j$).
		By construction, two vertices $x$ and $y$ of $\mathcal{J}$ are adjacent only if $x\in L_s'$ and $y \in L_t'$ such that $|s-t| \leq 1$ (mod $jq$).
		Also, any induced subgraph $\mathcal{J}[L_s' \cup \dots \cup L_{s+g-1}']$ of $\mathcal{J}$ is isomorphic to a subgraph $G[L_t' \cup \dots \cup L_{t+g-1}']$ of $G$ for some $t$. 
		Thus we see that $\mathcal{J}$ has girth $g$ and minimum degree $\delta$.
		The graph $\mathcal{J}$ also has order $\frac{jq}{g}M(\delta, g)$, since its vertex set is $j$ copies of $V$.
		All that remains is to show that $\eqt(\mathcal{J}) = jq$.
		By Theorem \ref{thm:lower_bound}, we have $\eqt(\mathcal{J}) \leq jq$, so it suffices to find an isometric $jq$-cycle.
		Let $C = w_0, \dots, w_{q-1}$ be an isometric $q$-cycle of $G$.
		Then
			\[
				C' = w_0^0, \dots, w_{q-1}^0, w_0^1, \dots, w_{q-1}^1, \dots \dots \dots, w_0^{j-1}, \dots, w_{q-1}^{j-1}
			\]
		is a cycle of length $jq$ in $\mathcal{J}$.
		The cycle $C'$ has one vertex in each set $L_i'$, and sets $L_s'$ and $L_t'$ have an edge between them only when $|s-t| \leq 1$.
		Therefore $C'$ is isometric, completing the proof.
	\end{proof}
	
	It's natural to ask whether Proposition \ref{prop:infinitely_many_q} can be `reversed': can we go from a large equatorial graph to a smaller one? 
	In the most extreme case, we would like to know if it is possible to derive a minimum degree $\delta$ girth $g$ Moore graph from such an equatorial graph, as this would completely determine the values of $\delta$ and $g$ for which there exists an equatorial graph.
	In general, it is not at all clear that this is possible. 
	The next Proposition gives a sufficient condition under which it is.
	
	\begin{prop}
		\label{prop:moore_from_equatorial}
		Let $G$ be an equatorial $(\delta+, g, q)$-graph. 
		If the partition of $G$ induced by its isometric $q$-cycles has a part $L_i$ with only a single vertex, then there exists a $\delta$-regular Moore graph with girth $g$.
	\end{prop}
	
	\begin{proof}
		Let $\{L_i : i \in [0,q-1]\}$ be the partition of $G$ induced by its isometric $q$-cycles and suppose without loss of generality that $L_0 = \{u\}$ for some vertex $u$. 
		Note that $G$ is $\delta$-regular by Proposition \ref{prop:equatorial_regular}.
		By Lemma \ref{lem:periodic_li}, we have $L_g = \{v\}$ for some vertex $v$.
		Let $H = G[L_0 \cup \dots \cup L_g] \big/ {\sim}$ be the quotient graph formed from $G[L_0 \cup \dots \cup L_g]$ by identifying $u$ and $v$.
		We denote the new identified vertex of $H$ by $\tilde{u}$, and reserve labels $u$ and $v$ for the distinct vertices of $G$.
		We claim that $H$ is a Moore graph.
		By Lemmas \ref{lem:extremal_disk_counting} and \ref{lem:disk_in_interval}, we have that $|L_0 \cup \dots \cup L_g| = M(\delta, g) + 1$, so $H$ has order $M(\delta, g)$.
		By Corollary \ref{cor:neighbors_in_li}, we have $N(u) = L_{-1} \cup L_1$ and $N(v) = L_{g-1} \cup L_{g+1}$.
		From Lemma \ref{lem:periodic_li}, we get $|L_1| = |L_{g+1}|$. 
		Since both $u$ and $v$ have degree $\delta$ in $G$, we conclude that $|L_{g-1}| = \delta - |L_1|$.
		Therefore $N(\tilde{u}) = L_1 \cup L_{g-1}$ contains $\delta$ vertices. 
		Because the other vertices of $H$ also have degree $\delta$, $H$ is $\delta$-regular.
		Since $d_G(u,v) = g$ and $G$ has girth $g$, we have that the girth of $H$ is at least $G$.
		Any cycle in $H$ formed by identifying $u$ and $v$ in a $u-v$ geodesic has length exactly $g$, so $H$ has girth at most $g$.
		Thus $H$ is a $\delta$-regular graph with girth $g$ and $M(\delta, g)$ vertices, completing the proof.
	\end{proof}
	
	\section{Characterizing Equatorial Graphs with Small Girth}
	\label{sec:low_girth_char}
	
	In this section, we characterize all equatorial graphs with girth 3 or girth 4. 
	We also characterize the equatorial graphs with girth 5 and degree 3. 
	The structure of these equatorial graphs closely mirrors the structure of the corresponding Moore graphs with the same girth and degree. 
	For girths 3 and 4, the characterization is simple to establish from Theorem \ref{thm:equatorial_structure}.
	However, the girth 5 degree 3 case is substantially more involved and the proof uses ad hoc arguments that seem unlikely to scale to larger girths and degrees. 
	
	\begin{thm}
		\label{thm:girth_3}
		Let $G = (V,E)$ be an equatorial $(\delta, 3, q)$-graph.
		Then $q \equiv 0 \pmod 3$ or $\delta \equiv 2 \pmod 3$. 
		The vertex set $V$ can be partitioned into $q$ parts $L_0, \dots, L_{q-1}$ such that each part induces a clique. 
		The edge set $E$ consists of the clique edges, and every edge of the form $uv$ where $u\in L_i$ and $v\in L_{i+1}$ (subscripts mod $q$).
		\begin{itemize}[noitemsep]
			\item If $q \equiv 0 \pmod 3$, then there are 3 positive integers $n_0, n_1, n_2$ such that $n_0 + n_1 + n_2 = \delta + 1$ and $|L_i| = n_j$ when $i \equiv j \pmod 3$.
			\item If $q \not\equiv 0 \pmod 3$, then $|L_i| = \frac{\delta + 1}{3}$ for all $i$.
		\end{itemize} 
	\end{thm}
	
	\begin{proof}
		Since $G$ is equatorial it has order $n = \frac{q(\delta + 1)}{3}$. 
		Because the order $n$ is an integer, 3 divides $q$ or $3$ divides $\delta + 1$ so we have $q \equiv 0 \pmod 3$ or $\delta \equiv 2 \pmod 3$.
		Use Theorem \ref{thm:equatorial_structure} to partition $V$ into $q$ parts $L_0, \dots, L_{q-1}$.
		In particular, note for any $i$ that we have $|L_{i-1}| + |L_{i}| + |L_{i+1}| = \delta + 1$ and $|L_i| = |L_{i+3}|$.
		Since any vertex $u\in L_i$ has degree $\delta$, the vertex $u$ is adjacent to every vertex in $L_{i-1} 
		\cup L_{i} \cup L_{i+1}$.
		Thus each set $L_i$ induces a clique, and $E$ contains every edge of the form $uv$ where $u\in L_i$ and $v\in L_{i+1}$.
		The remainder of the theorem follows immediately from the fact that $|L_j| = |L_{j+3}|$ for all $j$.
	\end{proof}
	
	\begin{thm}
		\label{thm:girth_4}
		Let $G = (V,E)$ be an equatorial $(\delta, 4, q)$-graph.
		Then $q \equiv 0 \pmod 4$ or $\delta \equiv 0 \pmod 2$. 
		The vertex set $V$ can be partitioned into $q$ parts $L_0, \dots, L_{q-1}$, each of which is an independent set. 
		The edge set $E$ consists of every edge of the form $uv$ where $u\in L_i$ and $v\in L_{i+1}$ (subscripts mod $q$).
		\begin{itemize}[noitemsep]
			\item If $q \equiv 0 \pmod 4$, then there are 4 positive integers $n_0, n_1, n_2, n_3$ such that $n_0 + n_2 = \delta = n_1 + n_3$ and $|L_i| = n_j$ when $i \equiv j \pmod 4$.
			\item If $q \not\equiv 0 \pmod 4$, then $|L_i| = \frac{\delta}{2}$ for all $i$.
		\end{itemize} 
	\end{thm}
	
	\begin{proof}
		Because $G$ is equatorial, it has order $n = \frac{q}{2}\delta$.
		Let $L_0, \dots, L_{q-1}$ be the partition of $V$ guaranteed by Theorem \ref{thm:equatorial_structure}.
		Note that both $|L_{i-1}| + |L_{i}| + |L_{i+1}| + |L_{i+2}| = 2\delta$ and $|L_i| = |L_{i+4}|$ for all $i$.
		
		We claim that each set $L_i$ is independent. 
		Assume to the contrary that there exist a pair $u$ and $v$ of adjacent vertices in some set $L_i$.
		If $x \in L_{i-1}$ and $y \in L_{i+1}$, then $L_i = \dd_1(x) \cap \dd_1(y)$.
		Thus $x, u, v$ is a 3-cycle, contradicting the girth of $G$.
		Thus $L_i$ is independent.
		
		Let $u \in L_{i}$ and $v \in L_{i+1}$. 
		Since both $L_i$ and $L_{i+1}$ are independent, we have that $N(u) \subseteq L_{i-1} \cup L_{i+1}$ and $N(v) \subseteq L_{i} \cup L_{i+2}$.
		Because both $u$ and $v$ have degree $\delta$, and $|L_{i-1}| + |L_{i}| + |L_{i+1}| + |L_{i+2}| = 2\delta$, we conclude that $N(u) = L_{i-1} \cup L_{i+1}$ and $N(v) = L_{i} \cup L_{i+2}$.
		Thus $E$ consists of all edges $uv$ where $u\in L_i$ and $v\in L_{i+1}$.
		Further, $|L_i| + |L_{i+2}| = \delta$ and $|L_i| = |L_{i+4}|$ for every $i$, completing the proof when $q \equiv 0 \pmod 4$.
		
		Suppose that $q \not\equiv 0 \pmod 4$.
		We then have that $|L_i| = |L_{i+2}|$ and $|L_i| = \delta - |L_{i+2}|$ for all $i$ (subscripts taken mod $q$).
		Thus $|L_i| = \frac{\delta}{2}$ for all $i$, and so $\delta$ is even.
	\end{proof}
	
	In the proof of Theorem \ref{thm:sharp_moore_parameters}, the graph $\mathcal{F}(3, 5, 5j)$ is constructed by taking $j$ disjoint copies $H_1, \dots, H_j$ of a Petersen graph with some edge $uv$ removed, and adding all edges $v_1u_2, v_2u3, \dots v_qu_1$ (for example, see Figure \ref{fig:homomorphism}).
	
	\begin{thm}
		\label{thm:girth_5_delta_3}
		Let $G$ be an equatorial $(3, 5, q)$-graph.
		Then $q \equiv 0 \pmod 5$, and $G$ is isomorphic to $\mathcal{F}(3, 5, q)$.
	\end{thm}
	
	\begin{proof}
		Suppose that $G$ is an equatorial $(3, 5, q)$-graph, and note that $n = \frac{q}{5}(3^2 + 1) = 2q$.
		Let $L_0, \dots, L_{q-1}$ be the unique partition of $V$ from Theorem \ref{thm:equatorial_structure}.
		Since $n = 2q$, the average number of vertices of a part $L_i$ is $\frac{1}{q} \sum_{i=0}^q |L_i| = 2$.
		Thus there are two cases: either $|L_i| = 2$ for every $i$, or there exists some $i$ such that $|L_i| = 1$.
		
		First suppose that $|L_i| = 2$ for all $i$.
		Let $C = u_0, \dots, u_{q-1}$ be an isometric $q$-cycle of $G$ labeled so that $u_i \in L_i$.
		Let $v_i$ denote the unique vertex of $L_i - \{u_i\}$.
		We consider two subcases. 
		Subcase 1: there is some $i$ such that $u_i$ and $v_i$ are adjacent.
		Since $d(v_i) = 3$, the vertex $v_i$ has two neighbors in the set $L_{i-1} \cup L_{i+1}$.
		As $G$ has girth 5, $v_i$ is not adjacent to $u_{i-1}$ or $u_{i+1}$, so $v_i$ is adjacent to both $v_{i-1}$ and $v_{i+1}$.
		The vertices $v_{i+1}$ and $u_{i+1}$ both have degree 3. 
		Since $G$ has no 4-cycles, $v_{i+1}$ and $u_{i+1}$ are not adjacent.
		Thus $N(u_{i+1}) = \{u_i, u_{i+2}, v_{i+2}\}$ and $N(v_{i+1}) = \{v_i, v_{i+2}, u_{i+2}\}$. 
		However this induces a 4-cycle on $u_{i+1}, u_{i+2}, v_{i+1}, v_{i+2}$, contradicting the girth of $G$. 
		Subcase 2: for all $i \in [0, q-1]$, the vertices $u_i$ and $v_i$ are not adjacent.
		As $G$ has girth 5 and $d(v_i) = 3$, we have that $v_i$ is adjacent to $v_{i-1}, v_{i+1}$ and either $u_{i-1}$ or $u_{i+1}$. 
		Assume without loss of generality that $v_i$ is adjacent to $u_{i+1}$.
		Since $d(u_i) = 3$, the vertex $u_i$ is adjacent to either $v_{i-1}$ or $v_{i+1}$. 
		Thus $G$ contains a 4-cycle, contradicting that $G$ has girth 5.
		In both subcases we find a contradiction, so there exists some $i$ such that $|L_i| = 1$.
		
		By Lemma \ref{lem:periodic_li}, the part $|L_{i+5t}|$ also contains a single vertex for all $t\in \mathbb{Z}$ (subscripts mod $q$).
		If $q$ is not a multiple of 5, then $q$ and 5 are coprime and so $|L_j| = 1$ for all $j$, which is impossible.
		Thus $q \equiv 0 \pmod 5$.
		Denote $L_i = \{u\}$ and $L_{i+5} = \{v\}$.
		By the same argument as in the proof of Proposition \ref{prop:moore_from_equatorial}, the graph formed from $G[L_i \cup \dots \cup L_{i+5}]$ by identifying $u$ and $v$ is a Moore graph with $\delta = 3$ and $g = 5$.
		I.e., it is the Petersen graph.
		Since $N_G(v) = L_{i+4} \cup L_{i+6}$ and $d(v) = 3$, either $|L_{i+4}| = 1$ or $|L_{i+6}| = 1$.
		Assume without loss of generality that $|L_{i+4}| = 1$.
		Then $G[L_i \cup \dots L_{i+4}]$ is isomorphic to the Petersen graph $P$ with an edge $e$ removed. 
		As $q$ is a multiple of 5, the same argument shows that each induced subgraph of the form $G[L_{i+5t} \cup \dots \cup L_{i + 5t + 4}] = P - e$ ($t\in \mathbb{Z}$, subscripts mod $q$). 
	\end{proof}
	
	\section{Conclusion and further questions}
	\label{sec:conclusion}
	
	We have shown that $(\delta+, g, q)$-graphs satisfy a variant of the Moore bound, provided $q$ is sufficiently larger than $g$.
	We constructed graphs that show this bound is best possible when there exists a Moore graph with parameters $\delta$ and $g$.
	We call the `Moore graph analogues' attaining this lower bound equatorial graphs.
	We also bounded the order of $C_4$-free graphs with minimum degree $\delta$ and equator $q$, and showed this bound is almost sharp for many values of $\delta$.
	
	Theorem \ref{thm:equatorial_structure}, gives strong constraints on the structure of equatorial graphs.
	Like Moore graphs, equatorial graphs are regular, and every vertex lies on a maximum-length isometric cycle.
	We proved that every equatorial graph $G$ admits a highly structured unique partition induced by its isometric $q$-cycles --- and this partition can be realized as the partition induced by the fibers of a homomorphism from $G$ onto a cycle $C_q^\circ$ with a loop on each vertex.
	
	We used Theorem \ref{thm:equatorial_structure} to completely characterize the equatorial graphs with girth 3 and girth 4, as well the equatorial $(3,5)$-graphs.
	
	The most obvious question raised by the results above is this: do there exist equatorial graphs for combinations of $\delta$ and $g$ for which there is no Moore graph?
	The author conjectures that there do not.
	
	\begin{con}
		\label{conj:equatorial_iff_moore}
		There exists an equatorial graph with degree $\delta$ and girth $g$ if and only if there exists a Moore graph with degree $\delta$ and girth $g$.
	\end{con}
	
	Note that Theorem \ref{thm:sharp_moore_parameters} establishes the easy direction: from a Moore graph with degree $\delta$ and girth $g$ we can construct an equatorial $(\delta, g)$-graph.
	
	Our results beg another extremal question: what can be said of minimum-order $(\delta, g, q)$-graphs? 
	We call a $\delta$-regular graph with girth $g$ and equator $q$ that has minimum order a $(\delta, g, q)$-\textit{cage}.
	These $(\delta, g, q)$-cages generalize both cages and equatorial graphs, as every equatorial graph is a $(\delta, g, q)$-cage, and every $(\delta, g)$-cage is a $(\delta, g, q)$-cage for some $q$.
	
	\begin{ques}
		\label{ques:cage}
		What bounds can be found for the orders of $(\delta, g, q)$-cages?
	\end{ques}
	
	Remark \ref{rem:cage_examples} gives some upper bounds addressing Question \ref{ques:cage} --- but only for $q$ a multiple of $g$.
	We note that a similar variant of cages, those with a prescribed diameter, has been studied in \cite{araujo2024note}.
	
	Throughout the paper, and in our definition of equatorial graphs, we restricted our attention to graphs in which $q > 6k+3$. 
	If $q$ is sufficiently small compared to $g$, even our lower bound $n \geq \frac{q}{g}M(\delta, g)$ can fail.
	For example, the wheel $C_5 + K_1$ has $n = 6$, $q=5$ and $g = \delta = 3$.
	We thus ask for bounds on the order of a graph when $g < q \leq 6k+3$.
	
	\begin{ques}
		\label{ques:small_q}
		What is the smallest order of a graph with girth $g \in \{2k+1, 2k+2\}$, minimum degree $\delta$ and equator $g < q \leq 6k+3$?
	\end{ques}
	
	\section*{Acknowledgments}
	The author is grateful to David Erwin for helpful discussions.
	
	\bibliographystyle{amsplain}
	\bibliography{isocycles}{}
	
\end{document}